\documentclass{amsart}

\input{style.sty}

\usepackage[pagewise]{lineno}

\title{Degree-1 maps and rank inequalities in Heegaard Floer homology}

\author{Fraser Binns}
\address{Department of Mathematics, Princeton University}
\email{fb1673@princeton.edu}
\author{Sudipta Ghosh}
\address{Department of Mathematics, University of Notre Dame}
\email{sghosh7@nd.edu}


\thanks{FB was supported by the Simons Grant {\em New structures in low-dimensional topology}}

\begin{document}

\begin{abstract}Ghosh-Sivek-Zentner constructed degree-1 maps from certain rational homology solid tori to the twisted $I$-bundle over the Klein bottle. We show that these maps yield rank inequalities for Heegaard Floer homology. To do so, we use Hanselman-Rasmussen-Watson's immersed curve interpretation of bordered Floer homology, extending their proof of a similar rank inequality corresponding to degree-1 maps to the solid torus. Our result provides further evidence for Kronheimer-Mowka's conjectured relationship between Heegaard Floer homology and instanton Floer homology.
\end{abstract}
\maketitle
\section{Introduction}

Consider $M_1,M_2$ a pair of compact, connected oriented $n$-manifolds with boundary. Recall that  ${H_n(M_i, \partial M_i;\Z)\cong\Z}$ comes equipped with a canonical generator $[\omega_i]$. A \emph{degree-1 map}  is a continuous map $f: M_1 \to M_2$ such that $f_*([\omega_1])=[\omega_2]$. In general, degree-1 maps are difficult to construct or classify, but for surfaces there is a complete classification due to Edmonds~\cite[Theorem 3.1]{Edmondsdegree1}. The situation for $3$-manifolds is more complicated; see~\cite{MR4138738, MR3177908} for some relevant results. The degree-1 maps relevant for this paper are constructed in forthcoming work of Ghosh, Sivek and Zentner~\cite{GSZdegree-1}. To describe them, recall that the half-lives half-dies principle implies that, up to a change of orientation, there is a unique simple closed curve in $\partial M$ of finite order in $H_1(M;\Z)$. These curves are called the \emph{rational longitudes of $M$}. For each manifold with torus boundary with rational longitudes of order $n$, $M$, Ghosh-Sivek-Zentner construct a degree-1 map from $M$ to a manifold $E_n$ constructed as follows: consider the planar surface
\[ P_n := S^2 \setminus \bigsqcup_{i=1}^n \mathring{D}^2 \]
given by the $2$-sphere with $n$ disjoint open disks removed.  Taking these disks to be centered at evenly spaced points around the equator of $S^2$,  consider the homeomorphism
\[ \phi_n: P_n \to P_n \]
given by $\frac{2\pi}{n}$ rotation about the axis through the north and south poles, as shown in Figure~\ref{fig:monodromy-phi}.  Define
\[ E_n := \frac{P_n \times [0,1]}{(x,1) \sim (\phi_n(x),0)}, \]
i.e. $E_n$ is the mapping torus of $\phi_n$. These manifolds have also been studied by Boyer-Clay~\cite[Section 2.2.3]{Boyerclay} and Hansleman-Watson~\cite[Section 7]{HanselmanWatson}, who computed their bordered Floer homology. Ghosh-Sivek-Zentner's result is as follows:

\begin{figure}
\begin{tikzpicture}
\draw[fill=gray!5] (0,0) circle (2);
\draw[thin,densely dashed] (-2,0) arc (180:0:2 and 0.4);
\foreach \i in {131,59} {
  \draw[fill=white] (\i:2 and 0.4) circle (0.1);
}
\draw[thin,-latex] (0,-2.25) -- (0,2.5);
\draw[-latex] (0,2.25) ++(60:0.4 and 0.1) arc (60:-260:0.4 and 0.1);
\node[right] at (0.4,2.25) {\footnotesize$\frac{2\pi}{n}$};
\begin{scope}
\clip (0,0) circle (2);
\path[save path=\front] (0,0) circle (2) \foreach \i in {203,275,347} { (\i:2 and 0.4) circle (0.1) };
\draw[fill=gray!20, fill opacity=0.6, even odd rule, use path=\front];
\begin{scope}[even odd rule]
  \clip[use path=\front];
  \draw[thin] (-2,0) arc (180:360:2 and 0.4);
\end{scope}
\end{scope}
\end{tikzpicture}
\caption{The monodromy $\phi_n: P_n \to P_n$.}\label{fig:monodromy-phi}
\end{figure}
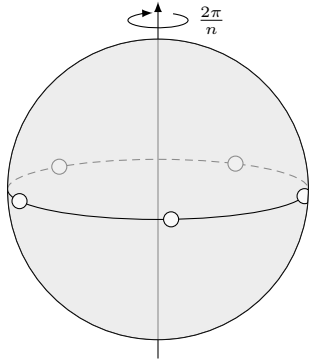

\begin{theorem}[\cite{GSZdegree-1}]\label{thm:pinch} 
Suppose $M$ is a compact, oriented 3-manifold with torus boundary and a rational longitude $\lambda$ of order $n$. 
Then there is a degree-1 map $f: M \to E_n  $
that restricts to a homeomorphism $\partial M \xrightarrow{\cong} \partial E_n$ sending $\lambda$ to a rational longitude of $E_n$.
\end{theorem}

The $n=1$ case of this theorem is the classical result that an integer homology solid torus $M$ admits a degree-1 map to the solid torus $S^1 \times D^2$. The $n=2$ case, where the target manifold is the twisted $I$-bundle over the Klein bottle, was proven in~\cite[Proposition~1.9]{ghosh2023rational}. Theorem~\ref{thm:pinch} can be used to produce degree-$1$ maps between certain closed toroidal $3$-manifolds as follows. If $Y= M \cup_{h} N$ is a $3$-manifold where $N$ has rational longitudes of order $n$, then $Y$ admits a a degree-1 map to $Y_1= M \cup_h E_n$ given by the identity on $M$ and $f$ on $N$. We call such a map an \emph{$n$-pinch}.

The relationships between degree-1 maps and classical invariants are well understood. In particular, a degree-$1$ map induces a surjection on fundamental groups. Consequently, if there is a degree-1 map $f \colon M \to N$, then every $\SU(2)$-representation of $\pi_1(N)$ lifts to an $\SU(2)$-representation of $\pi_1(M)$. Modulo a non-degeneracy assumption, the generators of the instanton Floer chain complex correspond to $\SU(2)$-representations of $\pi_1(Y)$. Therefore, in the instanton Floer setting, it is reasonable to expect that a degree-1 map imposes an inequality between the number of generators of the chain complexes of the domain and target. On the other hand, in~\cite[Conjecture 7.24]{kronheimer2010knots}, Kronheimer and Mrowka conjecture that a version of instanton Floer homology 
is isomorphic to a version of Heegaard Floer homology, an invariant due to Ozsv\'ath-Szab\'o~\cite{ozsvath2004holomorphic}:
\begin{equation} \label{Conj: KM}
 I^{\#}(Y;\C) \cong \widehat{\HF}(Y;\Z) \otimes \mathbb{C}.
 \end{equation} As is natural in light of the preceding discussion, 
 Hanselman, Rasmussen and Watson asked if degree-$1$ maps induce rank inequalities between the Heegaard Floer homologies of the corresponding manifolds~\cite[Question~1.13]{hanselman2016bordered}. 
They confirmed this conjecture for $1$-pinches~\cite[Theorem 1.12]{hanselman2016bordered}. Our goal in this paper is to verify the $n=2$ case:

\begin{thm}\label{Thm:main}
   If $M$ is obtained from a rational homology sphere $N$ by a $2$-pinch, then ${\rank(\widehat{\HF}(M))\leq \rank(\widehat{\HF}(N))}$.
\end{thm}

We prove this by refining Hanselman-Rasmusen-Watson's proof of the $n=1$ case. Specifically, we study the immersed curve invariants of rational homology solid tori with rational longitudes of order $2$, and show that the immersed curve invariant of the twisted $I$-bundle over the Klein bottle is the simplest possible invariant of such a manifold, in an appropriate sense. Here, the immersed curve invariant is a reinterpretation of bordered Floer homology for manifolds with torus boundary~\cite{bordered} due to Hanselman-Rasmussen-Watson~\cite{hanselman2016bordered}.

We conjecture that Theorem~\ref{Thm:main} holds more generally:
\begin{conjecture}\label{ref:conjecture}
    If $M$ is obtained from a rational homology sphere $N$ by an $n$-pinch, then ${\rank(\widehat{\HF}(M))\leq \rank(\widehat{\HF}(N))}$.
\end{conjecture}

Indeed, in Theorem~\ref{thm:precise}, we show that this more general theorem holds when we restrict the Heegaard Floer homology of $M$ and $N$ to certain $\spin^c$ structures. We defer the statement of this result to Section~\ref{sec:main}, as it is somewhat technical. See Remark~\ref{rem} for some further comments on Conjecture~\ref{ref:conjecture}.

This note is organized as follows. In Section~\ref{sec:background} we review relevant properties of Hanselman-Rasmussen-Watson's immersed curve interpretation of bordered Floer homology and some other background material. In Section~\ref{sec:main} we prove Theorem~\ref{Thm:main} and some other stronger results. 

\subsection*{Acknowledgments}
 The authors thank Tye Lidman for posing the question that led to this work, as well as Steven Sivek and Raphael Zentner. The first author thanks Liam Watson and Jonathan Hanselman for helpful conversations. The second author thanks Jake Rasmussen for helpful correspondences.

\section{Background}\label{sec:background}
We begin with some generalities. A \emph{rational homology solid torus} is a $3$-manifold with torus boundary, $M$, such that $H_*(M;\Q)\cong H_*(S^1;\Q)$. If $M$ is a rational homology solid torus, then $H_1(M)$ is a rank one $\Z$-module. Here, and throughout this paper unless explicitly stated otherwise, we take the singular (co)-homology groups to have coefficients in $\Z$. Examples of rational homology solid tori include the exteriors of knots in $S^3$ as well as the twisted $I$-bundle over the Klein bottle, as well as the manifolds $E_n$ discussed in the introduction more generally. Consider the map induced by the inclusion $i:\partial M\inj M$. The half-lives half-dies principle implies that the kernel of $i_*:H_1(\partial M;\Z)\to H_1(M;\Z)$ has rank one. $\ker(i_*)\subset H_1(M,\partial M;\Z)$ is generated by one of two element $h_M^\pm$ each of which we call  \emph{a homological longitude of $M$}. These can alternately be viewed as a single closed curve in $\partial M$ with a choice of one of two orientations. $h_M^\pm$ is a positive multiple of some primitive class $\lambda_M^\pm$ called~\emph{a rational longitude of $M$}. $\lambda_M^\pm$ can be viewed as a \emph{simple} closed curve in $\partial M$ equipped with one of two orientations.

A \emph{rational homology sphere} is a $3$-manifold $Y$ with $H_*(Y;\Q)\cong H_*(S^3;\Q)$. One way of producing rational homology spheres is to glue a pair of rational homology solid tori $M,N$ via a diffeomorphism $f:\partial M\to\partial N$ such that $f(\lambda_M^\pm)$ is \emph{not} isotopic to $\lambda_N^\pm$. This can be verified by an application of the Mayer-Veitoris exact sequence. As a partial converse, the exterior of any rationally null-homologous knot in a rational homology sphere is a rational homology solid torus.

\subsection{$\spin^c$-structures} By work of Turaev~\cite{turaev1990euler}, a \emph{$\spin^c$-structure} on a $3$-manifold $M$ can be viewed as an equivalence class of non-vanishing vector fields. Here, two non-vanishing vector fields are equivalent if they are isotopic in the complement of an embedded $3$-ball in $M$. There is an affine correspondence between $\spin^c(M)$ and $H^2(M)$. There is also an operation on $\spin^c$ structures structures called \emph{conjugation} that corresponds to sending a vector field $v$ representing $\s$ to the vector field $-v$. We let $-\s$ denote the conjugate of a $\spin^c$ structure $\s$.

Pick an essential simple closed curve $c\subseteq \partial M$. Such a choice gives rise to a sutured manifold structure on $(M,\gamma)$. See~\cite{gabai1983foliations} for details. A \emph{relative $\spin^c$-structure on $M$} is an equivalence class of non-vanishing vector fields that agree with a prescribed vector field $v$ on the boundary; see~\cite[Section 4]{juhasz2006holomorphic}. Here, two such vector fields are equivalent if they are isotopic in the complement of a properly embedded $3$-ball in $M$. A relative $\spin^c$-structure $\s$ has a relative Chern class, $c_1(\s,t)$, where $t$ is a trivialization of $v^\perp|_{\partial M}$ that, in our context, is determined by $c$; see~\cite[Section 4]{juhasz2006holomorphic} again for details.

There is a canonical map $q:\spin^c(M,\gamma)\to \spin^c(M)$.  It is not too hard to see that this map is surjective. We set $\spin^c(M,\gamma,\s):=\{\overline{\s}\in\spin^c(M,\gamma):q(\overline{\s})=\s\}$. There is an affine correspondence between relative $\spin^c$ structures and $H^2(M,\partial M)$. There is a map $H_2(M,\partial M)\to H_2(M)$ from the long exact sequence of the pair $(M,\partial M)$. It is not too hard to see that this map is also surjective. Indeed, the following diagram commutes:

 \begin{equation*}
\begin{tikzcd}
\spin^c(M, \gamma) \arrow[d,  two heads, "q"]\arrow[r,"\cong"]&H^2(M,\partial M)\arrow[d, two heads,"q^*"]\\ \spin^c(M)\arrow[r,"\cong"] &H^2(M)
\end{tikzcd}
\end{equation*}

\noindent where the horizontal maps are the affine correspondences discussed above.

\subsection{The immersed curve invariant}
Let $T_z$ denote the torus $\R/\Z\times\R/\Z$ with a marked point $z$ at $(0,0)$. There is an invariant of rational homology solid tori due to Hanselman-Rasmussen-Watson that takes value in the set of regular homotopy classes of immersed curves in the complement of $z$ equipped with local systems~\cite{hanselman2016bordered}. Since local systems are of little significance to this paper we do not discuss them here. Let $\gamma_M$ denote the immersed curve invariant of $M$.

The immersed curve invariant has a number of properties that will prove useful to us. $\gamma_M$ has a component $\gamma_{M,\s}$ for each $\s\in\spin^c(M)$. Suppose $M$ is a rational homology solid torus with rational longitudes of order $k$. Up to homotopy in $\R/\Z\times\R/\Z$ --- where we drop the requirement that the homotopy occurs in the complement of the basepoint --- $\gamma_{M,\s}$ is given by the curve $\R/\Z\to \R/\Z\times\R/\Z$ defined by $[x]\mapsto ([kx],0)$~\cite[Corollary 6.6]{hanselman2016bordered}.

\begin{example}
We describe the immersed curve of the maniffolds $E_{2k}$ discussed in the introduction in the two self-conjugate $\spin^c$ structures. This will be relevant for the proof of Lemma~\ref{lem:generallemma}. Hanselman-Watson computed (a version  of) the bordered Floer homology of $E_{k}$, $\widehat{\CFD}(E_k,\phi,\lambda)$, in~\cite[Section 7]{HanselmanWatson}. Here $\phi$ and $\lambda$ parameterize $\partial E_{k}$. The two relevant components are $(d_0^*)^t$ and $(a_k^*b^*_{-k})$\footnote{There is a typo in the statement of this result in~\cite[Section 7]{HanselmanWatson}.}, where here we follow loop calculus notation as described in~\cite[Section 3.2]{HanselmanWatson}. In terms of bordered Floer homology notation, these components are the type $D$ structures encoded by the graphs shown in Figure~\ref{fig:2} and Figure~\ref{fig:3}.

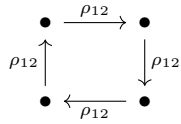
\begin{figure}[b]\centering
\begin{tikzcd}
    \bullet \ar[r,"\rho_{12}"]& \bullet \ar[d,"\rho_{12}"]\\\bullet \ar[u,"\rho_{12}"]&\bullet\ar[l,"\rho_{12}"]
\end{tikzcd}
   \caption{The component for $(d_0^*)^{4}$. More generally, the component for $E_n$ is a loop of $n$ arrows each decorated with $\rho_{12}$.}\label{fig:2}
      \end{figure}

     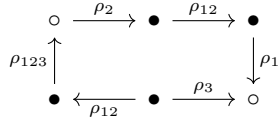
\begin{figure}[b]\centering
\begin{tikzcd}
   \circ \ar[r,"\rho_2"]& \bullet \ar[r,"\rho_{12}"]&\bullet \ar[d,"\rho_1"] \\\bullet \ar[u,"\rho_{123}"]&\bullet\ar[r,"\rho_3"]\ar[l,"\rho_{12}"]&\circ
\end{tikzcd}

   \caption{The component for $(a_2^*)(b_{-2}^*)$ in $E_4$. The corresponding component for $E_{2n}$ can be obtained by adding an extra $n-2$ $\rho_{12}$ arrows in the middle of the upper and lower rows of the diagram.}\label{fig:3}
   \end{figure}

The algorithm for reinterpreting $\widehat{\CFD}(E_{2k},\phi,\lambda)$ in terms of immersed curves is described in~\cite[Section 2.4]{hanselman2016bordered}. Applying this algorithm, we find that the two components of the immersed curve invariant are given as shown in Figure~\ref{fig:E2nf0} and Figure~\ref{fig:E2nfn}.

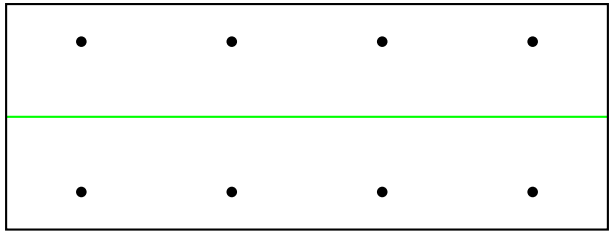
\begin{figure}[b]\centering

    \begin{tikzpicture}

     \fill[black] (-1, 1) circle (2pt);
     \fill[black] (1, 1) circle (2pt);
      \fill[black] (-3, 1) circle (2pt);
          \fill[black] (3, 1) circle (2pt);
               \fill[black] (-1, -1) circle (2pt);
     \fill[black] (1, -1) circle (2pt);
      \fill[black] (-3, -1) circle (2pt);
          \fill[black] (3, -1) circle (2pt);

   \draw[green, thick]  (-4,0) -- (4,0);
             
        \draw[thick]  (-4,-1.5) rectangle (4,1.5);
   
    \end{tikzpicture}
   \caption{A component of the immersed curve invairant of $E_4$ corresponding to a self-conjugate $\spin^c$-structure.}\label{fig:E2nf0}
  \end{figure}

  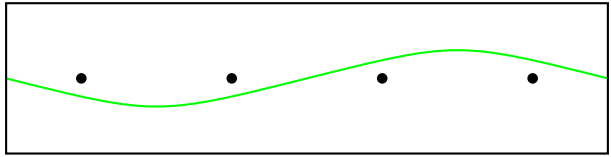
\begin{figure}[b]\centering

    \begin{tikzpicture}

     \fill[black] (-1, 0) circle (2pt);
     \fill[black] (1, 0) circle (2pt);
      \fill[black] (-3, 0) circle (2pt);
          \fill[black] (3, 0) circle (2pt);

   \draw[green, thick]  (-4,0) .. controls (-2, -0.5) .. (0, 0);
   \draw[green, thick]  (0,0) .. controls (2, 0.5) .. (4,0);
             
        \draw[thick]  (-4,-1) rectangle (4,1);
   
    \end{tikzpicture}
   \caption{A component of the immersed curve of $E_4$ corresponding to a self-conjugate $\spin^c$-structure. To obtain the corresponding component of the immersed curve for $E_{2n}$, replace the two basepoints above (below) the horizontal row of black dots with $n$ evenly placed black dots.}\label{fig:E2nfn}

\end{figure}
\end{example}

  If $M$ and $N$ are rational homology solid tori with immersed curve invariants $\gamma_M$ and $\gamma_N$ respectively and $M\cup_\psi N$ is a rational homology sphere obtained by gluing $M$ and $N$ by some diffeomorphism $\psi:\partial M\to \partial N$, then~\cite[Theorem 1.12]{hanselman2016bordered} tells us that \begin{align}\label{eq:Hfcompute}
      \rank(\widehat{\HF}(M\cup_\psi N))=\rank(\HF(\gamma_M,\psi^*(\gamma_N)))\geq i(\gamma_M,\psi^*(\gamma_N)).
  \end{align}

  Here $\HF(-,-)$ is a version of Lagrangian Floer homology, $i(\alpha,\beta)$ is the geometric intersection number of $\alpha$ and $\beta$, and $\psi^*$ is a diffeomorphism of the solid torus determined by $\psi$. Here, and throughout this paper, we take $\widehat{\HF}(-)$ to have coefficients in $\Z/2$.

Following~\cite[Section 7.1]{hanselman2016bordered}, we find it helpful to equip $T_z$, with one of a family of complete Riemannian metrics parameterized by $\epsilon$. This Riemannian metric is flat outside of an $\epsilon$ neighborhood of $z$, $\nu_\epsilon(z)$. We refer the reader to~\cite[Section 7.1]{hanselman2016bordered} for details. We denote $ T_z$ equipped with the Riemannian metric corresponding to $\epsilon$ by $(T_z,\epsilon)$, though we will usually suppress the $\epsilon$ dependence in our notation. In the proof of Theorem~\ref{Thm:main}, we will often appeal to the following technical result.

\begin{lemma}[{\cite[Lemma 7.1]{hanselman2016bordered}}]
Every non-trivial homotopy class of curve in $T_z$ can be isotoped to a geodesic in $(T_z,\epsilon)$. This geodesic representative is unique unless it is either a straight line in the complement of a $2\epsilon$-neighborhood of $z$, or parallel to $\partial(\nu_{\epsilon}(z))$. Moreover, distinct pairs of geodesics intersect transversely and minimally.\end{lemma}

As is often the case when working with immersed curves, we find it helpful to lift them to a specific cover of $T_z$, namely $\overline{T_z}:= \R/\Z\times\R$. The lift depends on the $\spin^c$ structure; i.e. we cannot simultaneously lift all components of the immersed curve to $\overline{T_z}$ in a satisfactory way.  Under this interpretation, relative $\spin^c$-structures are in correspondence with line segments $\{0\}\times(l+\frac{1}{2},l-\frac{1}{2})$, where $l\in\frac{1}{2}\Z$. The advantage of this approach for us is that the value of the \emph{Alexander grading} of $\overline{\s}$, $A:=\dfrac{\langle c_1(\overline{\s},t),[\Sigma]\rangle}{2}$, is exactly $kl$, where $k$ is the order of the rational longitude of the rational longitude of~\cite[Proposition 56]{hanselman2018heegaard}. Note that for a self-conjugate $\spin^c$-structure, the lift is determined by the property that the minimum height intersection point should be minus the maximum height intersection point.

Equation~\ref{eq:Hfcompute} can be reinterpreted as a result for curves in $\overline{T_z}$. Specifically, for each $\s\in\spin^c(M)$ we may lift $\gamma_{\s}$ to $\overline{T_z}$ and then compute an appropriate version of Lagrangian intersection Floer homology with the union of all the lifts of $\gamma_N$ to $\overline{T_z}$.

\section{Proof of the main Theorem and its generalizations}\label{sec:main}

In this section we prove Theorem~\ref{Thm:main}. Throughout this section we let $M$ denote a rational homology solid torus. We let $\lambda_M$ denote a fixed rational longitude of $M$, and $\gamma_{\s}$ denote the component of the immersed curve invariant of $M$ corresponding to the $\spin^c$-structure $\s$. We also fix a choice of dual curve to $\lambda_M$ in $\partial M$, with the property that $\lambda_M\cdot\mu_M=1$.  We begin with a lemma in singular homology.

\begin{lemma}\label{lem:lem31}
    Suppose $M$ is a rational homology solid torus with rational longitudes of order  $k$. Fix an identification $H_1(M)\cong \Z\oplus T$ where $T$ is torsion. Then the class $(1,0)$ is $k$-torsion in $H_1(M,\partial M)$.
\end{lemma}

\begin{proof}
    
Fix an isomorphism $H_1(M)\cong \Z\oplus T$. Observe that under the map induced by the inclusion $\partial M\to M$ on homology, $H_1(\partial M)\to H_1(M)$, we have that $\lambda_M\mapsto (0,\alpha)$, $\mu_M\mapsto (n,\beta)$  for some $\beta$, where $\alpha$ is of order $k$, and $n$ is some integer. Note that $H_1(M,\partial M)\cong H_1(M)/H_1(\partial M)$. Observe that $H_1(M,\partial M)$ is torsion, and indeed that $|H_1(M,\partial M)|=\dfrac{|n| \cdot \text{ord}(\beta)|T|}{k}$.

Consider the long exact sequence of the pair $(M,\partial M)$

    \begin{equation}
\begin{tikzcd}H_3(M,\partial M)\ar[r,"\partial"]&H_2(\partial M)\ar[r,"i_*"]& H_2(M)\ar[r, "q_*"]& H_2(M,\partial M)\ar[r, "\partial"]&H_1(\partial M).\end{tikzcd}\end{equation}

Observe that by Poincar\'e-Lefshetz  duality and the universal coefficient theorem $$H_2(M)\cong H^1(M,\partial M)\cong H_1(M,\partial M)\cong \Z\oplus T/((0,\alpha),(n,\beta)).$$ Likewise $$H_2(M,\partial M)\cong H^1(M)\cong H_1(M)\cong \Z\oplus T.$$ Since $H_1(\partial M)\cong\Z\oplus\Z$, it follows that $\partial:H_2(M,\partial M)\to H_1(M)$ kills the torsion and consequently $q_*$ yields a surjection $$\Z\oplus T/((0,\alpha),(n,\beta))\surj T.$$ Indeed, the map $\partial:H_3(M,\partial M)\to H_2(\partial M)$ is given by $\PD^{-1}\circ i^*\circ \PD:$, where $i^*:H^0( M)\to H^{0}(\partial M)$, which is an isomorphism. Consequently $q_*$ is injective and restricts to an isomorphism $\Z\oplus T/((0,\alpha),(n,\beta))\to T$. In particular, $\dfrac{|n|\cdot\text{ord}(\beta)|T|}{k}= |T|$, so that $k=|n|\cdot\text{ord}(\beta)$. It follows that $(1,0)$ generates a $\Z_{k}$ summand of $H_1(M,\partial M)\cong \Z\oplus T/((0,\alpha),(n,\beta))$.
\end{proof}

Now fix a surface $\Sigma\subset M$ with $[\partial\Sigma_M]=k[\lambda_M]\in H_1(\partial M)$. Recall that the choice of $\mu_M$ specifies a sutured structure, $(M,\gamma)$, on $M$, where the sutures are given by small neighborhoods of two parallel oppositely oriented copies of $\mu_M$.

We define a map $f_M:\spin^c(M)\to\Z/k$ by \begin{align*}\s\mapsto \left[\dfrac{\langle c_1(\overline{\s},t),[\Sigma_M]\rangle}{2}\right]_k\end{align*} where $\overline{\s}$ is any element of $\spin^c(M,\gamma,\s)$. Here, $[-]_k$ denotes reduction mod $k$. Apriori, it is not clear why the evaluation $\langle c_1(\overline{\s},t),[\Sigma_M]\rangle$ should be even, but this can be shown by observing that if $\s$ is a relative $\spin^c$-structure that maximises $\langle c_1(\overline{\s},t),[\Sigma_M]\rangle$ over relative $\spin^c$ structures for which $\SFH((M,\gamma),\s)\neq 0$, then  $$\langle c_1(\overline{\s},t),[\Sigma_M]\rangle =2g+2n-2$$ by~\cite[Theorem 4]{friedl2011decategorification}, where $g$ and $n$ are the genus and number of boundary components of a rational Seifert surface for $M$. The claim then follows from noting that for any other relative $\spin^c$ structure $\s'$, $c_1(\s',t)-c_1(\s,t)=2(\s'-\s)\in 2H^2(Y,\partial Y)$; see~\cite[Section 7]{OSdisksandapps}.

\begin{lemma}\label{lem:spincevaluation}
    Let $M$ be a rational homology solid torus with rational longitudes of order $k$. The map $f_M$ is well defined, surjective and $f_M(-\s)=-f_M(\s)$.

\end{lemma}
\begin{proof}

  Consider the following diagram:

    \begin{equation*}
\begin{tikzcd}
\spin^c(M, \gamma) \arrow[d,  two heads, "q"]\arrow[r,"\cong"]&H^2(M,\partial M)\arrow[d, two heads,"q^*"]\arrow[r,"\PD"]&H_1(M)\ar[d, two heads,"q_*"]\ar[r,"\cong"]&\Z\oplus T\ar[d, two heads]
\\ \spin^c(M)\arrow[r,"\cong"] &H^2(M)\arrow[r,"\PD"]&H_1(M,\partial M)\ar[r,"\cong"]&\Z\oplus T/\langle(1,\beta),(0,\alpha)\rangle
\end{tikzcd}
\end{equation*}

  Here we use the identifications and notation introduced in the proof of Lemma~\ref{lem:lem31}. the map ${H^2(M,\partial M) \surj H^2(M)}$ comes from the long exact sequence in cohomology of the pair $(M, \partial M)$, $q_*$ is the map from the long exact sequence in homology. The map $\spin^c(M,\gamma) \surj \spin^c(M)$ is the canonical one. The diagram clearly commutes. Now, if two relative $\spin^c$ structures $\overline{\s}^1$ and $\overline{\s}^2$ have the property that $q(\overline{\s}^1)=q(\overline{\s}^2)$ then $\overline{\s}^1-\overline{\s}^2$ lies in the span of $\PD[\mu_M]$ and $\PD[\lambda_M]$. Thus, if two relative $\spin^c$ structures $\overline{\s}^1$ and $\overline{\s}^2$ are lifts of a $\spin^c$ structure $\s$ then there is a pair of integers $a,b$ such that\begin{align*}\langle c_1(\overline{\s}^1_i,t),[\Sigma_M]\rangle - \langle  c_1(\overline{\s}^2_i,t),[\Sigma_M]\rangle&=\langle2(\overline{\s}^1-\overline{\s}^2),[\Sigma_M]\rangle\\&= \langle 2a\PD([\mu_M])+2b\PD([\lambda_M]),[\Sigma_M]\rangle\\
  &=\langle 2a[\mu_M],[\Sigma_M]\rangle\\&=2ka.\end{align*}

 \noindent It follows that $f_M$ is well defined.
  
We now show that $f_M$ is surjective. Observe that if $f_M(\s)=i$ for some $i$, then the relative $\spin^c$ structure $\s'$ with $\s'-\overline{\s}=(k,0)\in H^2(M,\partial M)\cong H_1(M)\cong\Z\oplus T$ has the property that 
  \begin{align*}\langle c_1(\overline{\s},t)-c_1(\s',t),[\Sigma_M]\rangle&=\langle 2k\PD(1,0),[\Sigma_M]\rangle\\&=2k.\end{align*}

  We now show that $f_M(-\s)=-f_M(\s)$. Suppose $\overline{\s}$ is represented by a vector field $v$. Then $-v$ represents $-\s$ as a $\spin^c$ structure. However, it doesn't represent \emph{any} relative $\spin^c$-structure, since it doesn't satisfy the necessary boundary conditions. Nevertheless, we can isotope $-v$ in a small neighborhood of $\partial M$ to a vector field $v'$ that represents a relative $\s'$ $\spin^c$ structure in $\spin^c(M,\gamma,-\s)$ as shown in Figure~\ref{fig:rotating}. In particular we see that $c_1(\s',t)=-c_1(\s,t)$, so the result follows.\end{proof}
  
  \begin{figure}
\centering
\begin{tikzpicture}
    \draw[thick] (0,0) circle (2);

    \foreach \angle in {0,45,...,315} {
        \draw[thick, ->,red] ({2*cos(\angle)}, {2*sin(\angle)}) -- ++(0,-1);
    }
     \foreach \angle in {0,45,...,315} {
        \draw[thick, ->,blue] ({2*cos(\angle)}, {2*sin(\angle)}) -- ++(0,1);
    }
    \fill[green] (-2,0) circle (3pt);
    \fill[green] (2,0) circle (3pt);
\end{tikzpicture}
\caption{The circle is a copy of $\lambda_M$ in $\partial M$. The points shown in green are the cores of the sutures, which are isotopic to $\mu_M$. The trivialization $t$ points directly out of the page. The isotopy from $-v|_{\partial M}$, shown in red, to $v'|_{\partial M}$, shown in blue, is induced by rotating the circle --- i.e. a copy of $\lambda_M$ --- by 180 degrees.}\label{fig:rotating}
\end{figure}
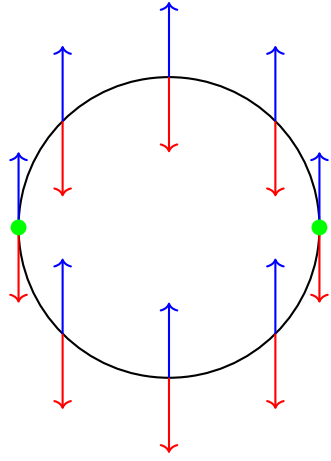

We now find a specific $\spin^c$-structure on $M$ satisfying some additional properties.

\begin{lemma}\label{lem:selfconjugate}
     Let $M$ be a rational homology solid torus with rational longitudes of order
$2k$ in $H_1(M)$ for some $k\in\Z$. Then there is a self-conjugate $\spin^c$-structure $\s$ with $f_M(\s)=k$.
\end{lemma}

In the proof we continue to use the identifications and notation set up in the proof of Lemma~\ref{lem:lem31}. 
\begin{proof}

Pick a self-conjugate $\spin^c$ structure $\s$ on $M$. Observe that since $f_M(-\s)=-f_M(\s)$ we have that either $f_M(\s)=k$ or $f_M(\s)=0$. If $f(\s)=k$ we are done.

On the other hand, if $f_M(\s)=0$, then we can pick a $\spin^c$ structure $\s'$ with $$\s'-\s=(k,0)\in \Z\oplus T/\langle(0,\alpha),(n,\beta)\rangle\cong H_1(M,\partial M)\cong H^2(M).$$ Observe that $\s'$ is still self-conjugate. If $\overline{\s}\in\spin^c(Y,\gamma,\s)$, $\overline{\s'}\in\spin^c(Y,\gamma,\s')$, we have that $$\overline{\s'}-\overline{\s}=(k,0)+a\lambda_M+b\mu_M$$ for $a,b\in\Z.$ Consequently \begin{align*}
    c_1(\overline{\s'},t)-c_1(\overline{\s},t)&=2(\overline{\s'}-\overline{\s},t)\\&=2(k,0)+2a\lambda_M+2b\mu_M,
\end{align*}

\noindent whence \begin{align*}
    \langle c_1(\overline{\s'},t),[\Sigma_M]\rangle&=\langle c_1(\overline{\s},t),[\Sigma_M]\rangle+\langle 2(k,0)+2a\lambda_M+2b\mu_M,[\Sigma_M]\rangle\\&=\langle c_1(\overline{\s},t),[\Sigma_M]\rangle+ 2k+4kb
\end{align*}

\noindent and 
\begin{align*}
    f_M(\s')&=\left[\dfrac{\langle c_1(\overline{\s'},t),[\Sigma_M]\rangle}{2}\right]_{2k}\\&=\left[\dfrac{c_1(\overline{\s},t),[\Sigma_M]\rangle}{2}+ k+2kb\right]_{2k}\\&=[f_M(\s)+k]_{2k}\\&=[k]_{2k}
\end{align*}

\noindent as desired.\end{proof}

 With the two previous lemmas at hand, we can now proceed to the immersed curve portion of our argument. We first fix some notation; let $\beta$ be the component of the immersed curve invariant of $E_{2k}$ corresponding to the $\spin^c$ structure $\s$ with $f_{E_{2k}}(\s)=k$.

\begin{lemma}\label{lem:generallemma}
Let $M$ be a rational homology sold torus with rational longitudes of order $2k$. Let $\gamma_\s$ be a component of the immersed curve invariant for $M$ corresponding to a self-conjugate $\spin^c$ structure with $f_M(\s)=k$. Let $\alpha$ be any immersed multi-curve in $T_z$. Then $i(\alpha,\beta)\leq i(\alpha,\gamma_\s)$.
\end{lemma}

We follow the general proof strategy of Hanselman-Rasmussen-Watson to prove the case corresponding to integer homology solid tori~\cite[Theorem 7.15]{hanselman2016bordered}.

\begin{proof}

   Consider the immersed curves $\gamma_\s,\alpha$ as described in the statement of the Lemma in the marked torus $T_z$. We describe an algorithm which produces a sequence of curves $(\gamma_i)$ beginning with $\gamma_{\s}$ and terminating with $\beta$, which has the property that $i(\gamma_i,\alpha)$ is a non-increasing function.

    The first step is to delete all components of $\gamma_\s$ that are null-homotopic in $T$ and replace every non-trivial local system on $\gamma_\s$ with the trivial one of the same dimension. Now, we fix a small $\epsilon$ and work in an {$\epsilon$-pegboard} diagram for $\overline{T_z}$ --- i.e. equip $\overline{T}_z$ with a Riemannian metric as discussed in Section~\ref{sec:background}. We claim that the conjugation symmetry of $\gamma_i$ fixes $\gamma_i$ set-wise. To see this, first note that since the conjugation symmetry is an isometry, the image of an $\epsilon$-geodesic is an $\epsilon$-geodesic. Now recall that an $\epsilon$-geodesic representative of a curve $\gamma_i$ is either unique or $\gamma_i$ is isotopic to a straight line in the flat part of $\overline{T_z}$~\cite[Lemma 7.1]{hanselman2016bordered}. But, given that $\gamma_i$ is isotopic to itself under the conjugation symmetry, if $\gamma_i$ were isotopic to a line of fixed slope in the flat part of $\overline{T_z}$, this would force the slope to be zero, and indeed that $\langle c_1(\overline{\s},t),[\Sigma_M]\rangle\equiv 0\mod 2k$, for all $\overline{\s}\in \spin^c(Y,\gamma,\s)$, a contradiction.

    Let $n(\gamma_i)$ denote the number of basepoints --- lifts of $z$ in $\overline{T_z}$ --- contained in the interval between the maximum and minimum height intersections $\gamma_i\cap(\{0\}\times\R)$. Note that $n(\gamma_i)$ is odd, since $${\dfrac{\langle c_1(\overline{\s},t),[\Sigma_M]\rangle}{2}\equiv k\mod 2k},$$ for all $\overline{\s}\in\spin^c(Y,\gamma,\s)$, so that the maximum and minimum height intersection points have height in the interval $(k+2Mk-\frac{1}{2},k+2Mk+\frac{1}{2})$ for some $M\in\Z$, and if two intersection points $\gamma_i\cap(\{0\}\times\R)$ are separated by $m$ basepoints on the line $\{0\}\times\R$, then the difference in the Alexander grading is $2km$.
    
    We obtain $\gamma_{i+1}$ by modifying $\gamma_i$ in such a way $n(\gamma_i)-n(\gamma_{i+1})$ is a non-negative even integer. Our algorithm terminates with a curve $\gamma_i$ such that $n(\gamma_i)=1$. It consists of applying four moves:
    \begin{enumerate}
       \item\label{move:resolve} resolving a self intersection in $\gamma_i$ to split of a closed component, 
        \item\label{move:delete}  deleting  components of $\gamma_i$ that are null-homotopic in $T$,
        \item\label{move:pegboard}  homotoping $\gamma_i$ to put it in $\epsilon$-pegboard position,
        \item\label{move:overpegs}  homotoping $\gamma_i$ through either the maximum or minimum height peg.
    \end{enumerate}

   Here by ``peg" we mean a lift of the basepoint $z$ to $\overline{T_z}$.  In applying these moves, we are always careful to respect the conjugation symmetry of $\gamma_i$: when we apply move~\ref{move:resolve} to an intersection point $x$, we also apply move~\ref{move:resolve} simultaneously at the image of $x$ under the conjugation symmetry.  When we apply move~\ref{move:delete} to remove a null-homotopic component $\gamma_i'$, we also apply move~\ref{move:delete} simultaneously to remove the image of $\gamma_i'$ under the conjugation symmetry. When we apply Move~\ref{move:overpegs} to isotope a curve segment $c$ over a peg $z$, we always apply Move~\ref{move:overpegs} to isotope the image of $c$ under the conjugation symmetry over the image of the peg $z$ under the conjugation symmetry. Move~\ref{move:pegboard} always results in a curve fixed by the conjugation symmetry, since $n(\gamma_i)$ by the same argument as used above.
    
   Note also that the first three moves do not increase $i(\alpha,\gamma_i)$. The final move may increase $i(\alpha,\gamma_i)$, so we have to apply it carefully, in a similar fashion to the corresponding steps in the proof of~\cite[Theorem 7.15]{hanselman2016bordered}, but where we must also account for the symmetry.

 Observe that after appropriate applications of move~\ref{move:resolve}, and move~\ref{move:delete}, we may assume that no component of $\gamma_i$ is homotopic to a component of $\alpha$. Thus, by an application of move~\ref{move:pegboard}, we may assume that $\alpha$ and $\gamma_i$ intersect minimally in $\epsilon$-pegboard position. Applications of move~\ref{move:resolve}, and move~\ref{move:delete} allow us to assume that the angle of $\gamma_i$ changes by at most $\pi$ between approaching and leaving any fixed peg.
\begin{figure}

\begin{subfigure}[b]{0.4\textwidth}\centering

    \begin{tikzpicture}

     \fill[black] (0, 0) circle (2pt);

           \draw[green, thick]  (-1,-0.5) .. controls (0, 0.5) and (0,0.5).. (1, -0.5);

             \draw[red, thick]  (-1,-0.9) .. controls (0, 0.5) and (0,0.5).. (1, -0.9);
        \draw[thick]  (-1,-1) rectangle (1,1);
   
    \end{tikzpicture}
   \caption{}\label{fig:typea}
      \end{subfigure}
      \begin{subfigure}[b]{0.4\textwidth}\centering

    \begin{tikzpicture}

     \fill[black] (0, 0) circle (2pt);

           \draw[green, thick]  (-1,-0.5) .. controls (0, 0.4) and (0,0.4).. (1, -0.5);

             \draw[red, thick]  (-1,-0.3) .. controls (0, 0.4) and (0.2,0.4).. (1, -0.9);
        \draw[thick]  (-1,-1) rectangle (1,1);
   
    \end{tikzpicture}
   \caption{}\label{fig:typeb}
      \end{subfigure}
      \begin{subfigure}[b]{0.4\textwidth}\centering

    \begin{tikzpicture}

     \fill[black] (0, 0) circle (2pt);

           \draw[red, thick]  (-1,-0.5) .. controls (0, 0.5) and (0,0.5).. (1, -0.5);

             \draw[green, thick]  (-1,-0.9) .. controls (0, 0.5) and (0,0.5).. (1, -0.9);
        \draw[thick]  (-1,-1) rectangle (1,1);
   
    \end{tikzpicture}
   \caption{}\label{fig:typec}
      \end{subfigure}
          \begin{subfigure}[b]{0.4\textwidth}\centering

    \begin{tikzpicture}

     \fill[black] (0, 0) circle (2pt);

           \draw[red, thick]  (-1,0.5) .. controls (0, -0.5) and (0,-0.5).. (1, 0.5);

             \draw[green, thick]  (-1,-0.9) .. controls (0, 0.5) and (0,0.5).. (1, -0.9);
        \draw[thick]  (-1,-1) rectangle (1,1);
   
    \end{tikzpicture}

    \caption{}\label{fig:typed}
    \end{subfigure}

\caption{Components of $\alpha$ and $\gamma_i$ near a maximal height peg. $\alpha$ is shown in red and $\gamma_i$ in green. Note that the $\alpha$-curves may wrap around the peg.}\label{fig:overpegs}

\end{figure}
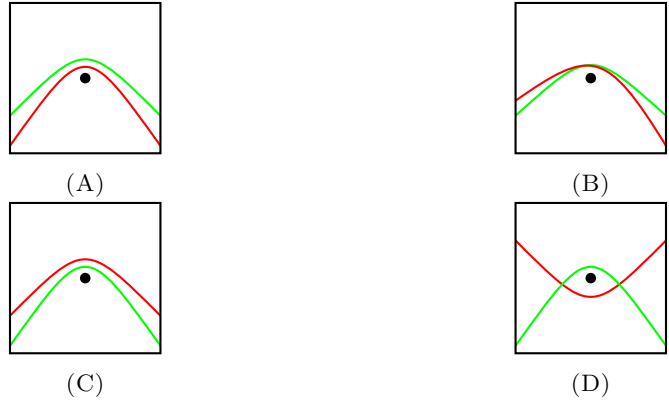
    
    We now study the behavior of $\gamma_i$ under an application move~\ref{move:overpegs} near the peg of height $n$ and a symmetric application of move~\ref{move:overpegs} near the peg at height $-n$.  Observe that in a neighborhood $N$ of the peg at height $n$, a component of $\gamma_i\cap N$ may take one of the four forms shown in~\cite[Figure 55]{hanselman2016bordered}, or Figure~\ref{fig:overpegs} above. Consider the conjugation symmetry of $\gamma_i$. This interchanges corners of type~\ref{fig:typea} and~\ref{fig:typed}, and interchanges corners of type~\ref{fig:typeb} and~\ref{fig:typec}. Now, observe that applying~\ref{fig:typec} at a peg of height $n$ and~\ref{fig:typeb} at a peg of height $-n$ has no effect on the intersection number. Likewise, while applying~\ref{fig:typea} increases the intersection number, simultaneously applying~\ref{fig:typea} at a peg of height $n$ and~\ref{fig:typed} at a peg of height $-n$ has no effect on the intersection number.

    We apply this algorithm until we are left with an immersed curve $\gamma_i$ supported in a neighborhood of $S^1\times\{0\}$ that contains only $\pi$-corners --- i.e. the slope of the curve is close to $0$. There are many such curves  --- for example, the components of the immersed curve of $E_M$, with the exception of the one depicted in Figure~\ref{fig:E2nf0} ---  however each can be transformed to the desired form --- i.e. as in Figure~\ref{fig:E2nfn} --- by further applications of move~\ref{move:overpegs}. To see this, consider the collection of intersection points $(\{0\}\times\R)\cap\gamma_i$. Observe that $2k=|\gamma_i\cap (\{0\}\times\R)|$. Endow $\gamma_i\cap (\{0\}\times\R)$ with a $\Z_{2k}$ ordering. Define a function $g:(\{0\}\times\R)\cap\gamma_i \to\{\pm1\}$ according to whether the intersection point lives above or below $S^1\times\{0\}$. If $g(x)=1$ for $0\leq x< k$, then we are done. If not, we can apply move~\ref{move:overpegs} to the intersection point $x$ corresponding to the first $0\leq i<k$ for which $g(x)=-1$, as well as to the intersection point that is the image of $x$ under the conjugation action, which is necessarily of the opposite sign. Applying this argument repeatedly results in the desired curve. \end{proof}

We now give various rank inequalities for Heegaard Floer homology. We begin by giving a mild reinterpretation of of~\cite[Theorem 7.15]{hanselman2016bordered}.  We also set some notation; given a decomposition $M=M_1\cup M_2$ and $\spin^c$ structures $\s_i\in\spin^c(M_i)$,   we set $\spin^c(M,\s_1,\s_2):=\{\spin^c(M):\s|_{M_i}=\s_i\}$.
\begin{theorem}[{\cite[Theorem 7.15]{hanselman2016bordered}}]\label{thm:restatement}
      Suppose $N$ is obtained from $M=M_1\cup M_2$ by a $k$-pinch of $M_2$ to $E_{k}$ for some $k$. Let $\s_0\in\spin^c(E_{k})$ be the $\spin^c$ structure with $f_{E_{k}}(\s_0)=0$, and $\s_{M_2}$ be any self-conjugate $\spin^c$ structure with $f_{M_2}(\s_{M_2})=0$. Then for every $\s_{M_1}\in\spin^c(M_1)$ we have that: \begin{align*}\rank\left(\underset{\s\in\spin^c(M,\s_{M_1},\s_{M_2})}{\bigoplus}\widehat{\HF}(M,\s)\right)\geq \rank\left(\underset{\s\in\spin^c(N,\s_{M_1},\s_{0})}{\bigoplus}\widehat{\HF}(N,\s)\right).\end{align*}
\end{theorem}
While this  formulation differs from --- and is indeed strictly weaker than ---  that presented in~\cite[Theorem 7.15]{hanselman2016bordered}, this formulation is a direct consequence of the proof of~\cite[Theorem 7.15]{hanselman2016bordered}.

We also have the following stronger but less general result;

\begin{theorem}\label{thm:precise}
    Suppose $N$ is obtained from $M=M_1\cup M_2$ by a $2k$-pinch of $M_2$ to $E_{2k}$ for some $k$. Let $\s_{k}\in\spin^c(E_{2k})$ be the $\spin^c$ structure $f_{E_{2k}}(\s)=k$ and $\s_{M_2}$ be any self-conjugate $\spin^c$ structure with $f_{M_2}(\s_{M_2})=k$. Then for every $\s_{M_1}\in\spin^c(M_1)$ we have that: \begin{align*}\rank\left(\underset{\s\in\spin^c(M,\s_{M_1},\s_{M_2})}{\bigoplus}\widehat{\HF}(M,\s)\right)\geq \rank\left(\underset{\s\in\spin^c(N,\s_{M_1},\s_{k})}{\bigoplus}\widehat{\HF}(N,\s)\right).\end{align*}
\end{theorem}

\begin{proof}
     Suppose $M=M_1\cup M_2$ and $N$ are as in the statement of the theorem.  We can write ${N=M_1\cup E_{2k}}$. Observe that since $M$ and $N$ are rational homology spheres, $\rank(\widehat{\HF}(M))$ and $\rank(\widehat{\HF}(N))$ can be computed as the geometric intersection number of $\gamma_{M_2}$ or $\gamma_{E_{2k}}$ with $\gamma_{M_1}$, where $\gamma_{M_1}$, $\gamma_{M_2}$   and $\gamma_{E_{2k}}$ are the immersed curve invariants for $M_1$, $M_2$, and  $E_{2k}$ respectively. $M_2$ has at least one self-conjugate $\spin^c$-structure $\s$ with $f^*(\s)=k$ by a combination of Lemma~\ref{lem:spincevaluation} and Lemma~\ref{lem:selfconjugate}. Let $\s_{k}$ be such a $\spin^s$ structure. Observe that for any of the $|H^2(M_1)|$ components of the immersed curve $\gamma_{M_1,\s}$ we have as a consequence of Theorem~\ref{Thm:main} that \begin{align*}i(\gamma_{M_1,\s},\gamma_k)>i(\gamma_{M_1,\s},\gamma_{\s_{k}}).\end{align*}

   Now Hanselman-Rasmussen-Watson's result~\cite[Theorem 1.12]{hanselman2016bordered} implies that ${\rank(\widehat{\HF}(Y_0,\s'))=i(\gamma_0,c_0)}$ while $\rank(\widehat{\HF}(Y_0,\s'))=i(\gamma_0,c_0)$, so the result follows.
\end{proof}

We now prove the rank inequality advertised in the introduction.
\begin{proof}[Proof of Theorem~\ref{Thm:main}]

Suppose $M$ and $N$ are in the statement of the theorem. Write $M=M_1\cup M_2$, as in the statement of Theorem~\ref{thm:precise}. Apply Theorem~\ref{thm:precise} to the $\spin^c$ structure $\s_1\in\spin^c(E_{2})$ with $f_{E_2}(\s_1)=1$ and every $\spin^c$ structure on $M_1$. Apply Theorem~\ref{thm:restatement} to the $\spin^c$-structure $\s_0$ on $E_2$ with $f_{E_2}(\s_0)=0$ and every $\spin^c$ structure on $M_1$. The result follows.\end{proof}

We conclude the paper with a remark on Conjecture~\ref{ref:conjecture}.

\begin{remark}\label{rem}
Of course, there is a version of the algorithm presented in the proof of Lemma~\ref{lem:generallemma} for $\gamma_\s\cup\gamma_{-\s}$ for an arbitrary $\spin^c$ structure $\s$ on an arbitrary rational homology solid torus. However, if $\s$ is, say, non self-conjugate, the authors see no reason why the algorithm would not terminate in a collection of loose curves. On the other hand, for $\s$ a non-self conjugate $\spin^c$ structure on the manifold $E_n$, $\gamma_\s$ is non-loose. It is therefore unclear to us how to prove a generalization of Lemma~\ref{lem:generallemma} and therefore, in turn, it is  unclear to us how to generalize the proof of Theorem~\ref{Thm:main}.
\end{remark}

\bibliographystyle{alpha}
\bibliography{bibliography}
\end{document}